\documentclass[smallextended,envcountsect]{svjour3}       
\usepackage{amsmath,amsfonts,amssymb,array,mathdots}
\usepackage{xcolor}

\smartqed  
\usepackage{graphicx}
\DeclareMathOperator{\Sol}{Sol}

\DeclareMathOperator{\AVVI}{AVVI}
\DeclareMathOperator{\AVI}{AVI}
\DeclareMathOperator{\pf}{pf}

\DeclareMathOperator{\Pa}{P}
\DeclareMathOperator{\Q}{Q}
\DeclareMathOperator{\Pone}{P_0}
\DeclareMathOperator{\Ptwo}{P_s}

\DeclareMathOperator{\LFVOP}{LFVOP}
\DeclareMathOperator{\VI}{VI}
\DeclareMathOperator{\ri}{ri}

\DeclareMathOperator{\N}{\mathbb{N}}
\DeclareMathOperator{\R}{\mathbb{R}}
\DeclareMathOperator{\F}{\mathcal{F}}
\DeclareMathOperator{\VVI}{VVI}
\DeclareMathOperator{\inte}{int}

\begin{document}

\title{Numbers of the connected components of the solution sets of monotone affine vector variational inequalities}

\author{Vu Trung Hieu}

\institute{Vu Trung Hieu \at Division of Mathematics, Phuong Dong University,\\ 171 Trung Kinh,  Cau Giay, Hanoi, Vietnam\\
E-mail: hieuvut@gmail.com}

\date{Received: date / Accepted: date}

\maketitle

\begin{abstract} This paper establishes several upper and lower estimates for the maximal number of the connected components of the solution sets of monotone affine vector variational inequalities. Our results give a partial solution to Question~2 in [N.D. Yen and J.-C. Yao, \textit{Monotone affine vector variational inequalities}, Optimization 60 (2011), pp. 53--68] and point out that the number depends not only on the number of the criteria but also on the number of variables of the vector variational inequality under investigation.
\end{abstract}

\section{Introduction}
The concept of \textit{vector variational inequality}, which was introduced by Giannessi \cite{G80} in 1980, has received a lot of attention from researchers. In particular, it has been shown that results on \textit{monotone affine vector variational inequalities} (monotone AVVIs) are very useful for studying linear fractional vector optimization problems and convex quadratic vector optimization problems (see, e.g., \cite{YY2011,Yen2012} and the references therein). Some open questions concerning the connectedness structure of the solution sets of monotone affine vector variational inequalities were raised in \cite{YY2011} and \cite{Yen2016}. 

Recently, thanks to several results from Real Algebraic Geometry, the investigation of the connectedness structure of the solution sets of polynomial vector variational inequalities has got a significant progresses \cite{HYY2015a,HYY2015b} (see also \cite{Hieu16}). The most important results assert that, under a mild constraint qualification, the proper Pareto solution sets, the Pareto solution sets, and the weak Pareto solution sets of these problems have finitely many connected components, which are path connected. Among other related results, we would like to mention the explicit upper bounds for the number of connected components in these solution sets that were given in \cite{Hieu16}. 

In this paper, we restrict our attention to monotone AVVIs. Several upper bounds for the number of connected components of the solution sets will be obtained. These numbers are strictly smaller than the upper bounds previously shown in \cite[Theorem 3.3]{Hieu16}.

About the maximal number of connected components of the solution sets of monotone AVVIs, there is a conjecture that this number equals the number of criteria in the vector variational inequalities (see \cite[Question 2]{YY2011}). However, by considering a skew-symmetric affine vector variational inequality, we will see that the conjecture  is wrong for bicriteria problems. Moreover, it will be proved that the number of the connected components depends on not only the number of criteria but also on the number of variables of the vector variational inequality. 

Section 2 collects the definitions, notations, and auxiliary results needed for our subsequent discussions. Section 3 studies the connectedness structure of the solution sets of bicriteria skew-symmetric affine vector variational inequalities. Section 4 focuses on the maximal number of connected components of the Pareto solution sets and the weak Pareto solution sets of monotone AVVIs.

\section{Preliminaries}
We now recall the notion of vector variational inequality and several elementary properties of the model.

The scalar product of $x, y$ from $\R^n$ is denoted by $\langle x,y\rangle$. Let $K\subset\R^n$ be a nonempty subset and $F_0: K\to\R^n$ a vector-valued function. The \textit{variational inequality} defined by $F_0$ and $K$ is the problem
$${(\VI)} \qquad {\text{Find}}\ \, x\in K \ \; \text{such that}\ \, \langle F_0(x),y-x\rangle\geq 0,\ \, \forall y\in K.$$
The corresponding solution set is denoted by $\Sol(\VI)$.
 Note that $x$ solves $(\VI)$ if and only if
$F_0(x)\in -N_K(x),$
where $N_K(x)$ is the normal cone  of $K$ at $x\in K$  which is defined by 
$$
N_K(x)=\{x^*\in \R^n:\langle x^*,y-x\rangle\leq 0, \forall y\in K\}.
$$
Clearly, if $x$ belongs to the interior of $K$, then $x\in\Sol(\VI)$ if and only if $F_0(x)=0$. So, when $K=\R^n$, $x$ solves $(\VI)$  if and only if $x$ is a zero point of the function $F_0$.

When $F_0$ is \textit{monotone}, i.e., $\langle  F_0(y)-F_0(x),y-x\rangle \geq 0$ for all $x,y\in K$, $(\VI)$ is called a \textit{monotone variational inequality}. Remind that if $F_0$  is monotone, then $\Sol(\VI)$ is a convex set \cite{FaPa03}.

Suppose that $F_l: K\to\mathbb R^n$ ($l=1,\dots, m$) are given vector-valued functions. We put $F=(F_1,\dots, F_m)$ and 
\begin{center}
	$F(x)(u)=( \langle F_1(x),u\rangle, \cdots, \langle F_m(x),u\rangle), \ \forall x\in K,\ \forall u\in\R^n$.
\end{center}
Let $C=\R^m_+$ and $\Delta_m=\left\{(\xi_1,\dots,\xi_m)\in \R^m_+\,:\,
\displaystyle\sum_{l=1}^m\xi_l=1\right\}$, where $\R^m_+$ is the nonnegative orthant of  $\R^m$. The relative interior of $\Delta_m$ is described by the formula ${\ri}\Delta_m=\{\xi\in\Delta_m\,:\, \xi_l>0,\ l=1,...,m\}.$ 
The problem
$${ (\VVI)} \qquad {\text{Find}}\ \, x\in K \ \; \text{such that}\ \, F(x)(y-x)\nleq_{C\setminus\{0\}}0,\ \, \forall y\in K,$$ 
is said to be the \textit{vector variational inequality} defined by $F$ and $K$. Here the inequality means that
$F(x)(x-y)\notin C\setminus\{0\}$. 
The solution set of $(\VVI)$ is denoted by $\Sol(\VVI)$ and called the \textit{Pareto solution set}. One associates to $(\VVI)$ the problem
$${(\VVI)^w} \qquad {\text{Find}}\ \, x\in K \ \; \text{such that}\ \, F(x)(y-x)\nleq_{\inte C}0, \ \, \forall y\in K,$$
where ${\inte}C$ is the interior of $C$ and the inequality means $F(x)(x-y)\notin {\inte}C$. The solution set of $(\VVI)^w$ is denoted by $\Sol^w(\VVI)$ and called the \textit{weak Pareto solution set}. One says that $(\VVI)$ is a \textit{monotone vector variational inequality} if all the operators $F_l, l = 1, ..., m,$ are monotone.

For each $\xi\in\Delta_m$, consider the variational inequality 
$$(\VI)_{\xi} \quad\  \text{Find }  x\in K \ \text{ such that }  \left\langle  \sum_{i=1}^m\xi_iF_i(x),y-x\right\rangle  \geq 0, \ \, \forall y\in
K.$$
The solution sets of  $(\VVI)$ can be computed or estimated via certain unions of the solution sets of $(\VI)_{\xi}$ with $\xi\in \Delta$. Those well-known scalarization formulas are as follows.

\begin{theorem} {\rm (see \cite{LKLY98}, \cite{LY01})}\label{thm_scalarization} It holds that
	\begin{equation}\label{scalarization}\bigcup_{\xi\in{\ri}\Delta_m}{\Sol(\VI)}_\xi\subset {\Sol(\VVI)}\subset {\Sol^w}{(\VVI)}=\bigcup_{\xi\in\Delta_m}{\Sol(\VI)}_\xi. \end{equation} 
	If $K$ is a polyhedral convex set, i.e.,
	$K$ is the intersection of finitely many closed half-spaces of
	$\mathbb R^n$, then the first inclusion in \eqref{scalarization}
	holds as equality.
\end{theorem}

The multifunction $\Phi:\Delta_m \rightrightarrows \R^n$, which is said to be the
\textit{basic multifunction} associated to the problem $(\VVI)$, is defined by setting
$\Phi(\xi)=	\Sol(\VI)_{\xi}$ for all $\xi\in\Delta_m$. If $K$ is polyhedral convex, then from Theorem \ref{thm_scalarization} it follows that $\Sol^{w}(\VVI)=\Phi(\Delta_m)$ and $\Sol(\VVI)=\Phi(\ri\Delta_m)$. So, one can use the basic multifunction $\Phi$ to investigate different properties of the solution sets of $(\VVI)$. 

When $K$ is a polyhedral convex set and $F_0(x)=Mx+q$ for all $x\in K$,
where $M\in\mathbb R^{n\times n}$ and
$q\in\mathbb R^n$, then $(\VI)$ is an {\it affine variational inequality}. In that case, the variational inequality problem and its solution set are denoted by
$(\AVI)$ and $\Sol(\AVI)$, respectively.
Note that, if $M\in\mathbb R^{n\times n}$
is a positive semidefinite matrix, i.e., $\langle Mv,v\rangle\geq 0$
for all $v\in\mathbb R^n$, then the affine operator $F_0(x)=Mx+q$
(with $q\in\mathbb R^n$ being fixed) is monotone on $K$. The
converse is true if ${\rm int}K\neq\emptyset$. If $M^T=-M$, where $M^T$ is the transpose of $M$, then one says that $M$ is \textit{skew-symmetric}. Clearly, a skew-symmetric matrix is positive semidefinite.  

Let $K$ be a polyhedral convex set given by $K=\{x\in \R^n:Ax\geq b\}$, where  $A=(a_{ij})\in {\mathbb R}^{p\times n}$ and $b=(b_i)\in {\mathbb R}^p$. By using Lagrange multipliers, one can characterize the solutions of $(\AVI)$ as follows. 

\begin{theorem}{\rm (see, e.g., \cite[Theorem~5.3]{LTY2005})}\label{Lagrange multipliers rule} Vector $x\in \R^n$ is a solution of $(\AVI)$  if and only if there exists $\lambda\in \R^p$ such that
	\begin{equation}\label{system}
\left\lbrace 	
	\begin{array}{l}	Mx-A^T\lambda+q =0, \\ 
	\lambda^T(Ax-b)=0, \ \lambda \geq 0 , \ Ax \geq b. 	\end{array}\right. 
	\end{equation}
\end{theorem}
\begin{remark}\label{remark_intK} If $x$ belongs to the interior of $K$, then $x$ is a solution of $(\AVI)$ if and only if $x$ solves the equation $Mx+q=0$. Consequently, if $(\AVI)$ is an unconstrained problem, then $$\Sol(\AVI)=\left\lbrace x\in \R^n: Mx+q=0\right\rbrace .$$
Moreover, if $M$ is a nondegenerate matrix, then $x=-M^{-1}q$ is the unique solution of $(\AVI)$, where $M^{-1}$ is the inverse of $M$.
\end{remark}

One says that $(\VVI)$ is an {\it affine} vector variational inequality 
if $K$ is a polyhedral convex set and there exist matrices
$M_i\in\mathbb R^{n\times n}$ and vectors $q_i\in\mathbb R^n$
$(i=1,\dots,m)$ such that $F_i(x)=M_ix+q_i$ for all $i=1,\dots,m$
and $x\in K$. In that case, the problem is denoted
by $(\AVVI)$. Herein, the parametric problem $(\AVI)_{\xi}$ of $(\AVVI)$ is defined by  $$F_{\xi}(x)= \displaystyle\sum_{i=1}^m\xi_iF_i(x)=M_{\xi}x+q_{\xi},$$ where $M_{\xi}=\displaystyle\sum_{i=1}^m\xi_iM_i$ and $q_{\xi}=\displaystyle\sum_{i=1}^m\xi_iq_i$. 
The problem $(\AVVI)$ is said to be \textit{monotone} if all the operators $F_i(x)=M_ix+q_i$ ($i=1,\dots,m$) are monotone. If all the matrices $M_i$ are skew-symmetric, then we will say that $(\AVVI)$ is a \textit{skew-symmetric} affine vector variational inequality.

To every index set $\alpha\subset I$, where $I:=\{1,\dots,p\}$, we associate  the following {\em pseudo-face}
$$\mathcal{F}_\alpha:=\left\{x\in {\mathbb R}^n\,:\, \sum_{j=1}^{n}a_{ij}x_j=b_i\ \,
\forall i\in\alpha,\ \, \sum_{j=1}^{n}a_{ij}x_j> b_i\ \,\forall i\notin\alpha\right\}$$ of $K$, where $a_{ij}$ is the element in the $i$-th row and the $j$-th column of $A$, and $b_i$ denotes the $i$-th component of $b$.
Clearly, $\F_{\emptyset}=\inte K$, $K=\displaystyle\bigcup_{\alpha\subset I}\F_{\alpha}$ and $\F_{\alpha}\cap \F_{\alpha'}=\emptyset$ with $\alpha \neq \alpha'$. 
Since $\Sol(\AVVI)\subset\Sol^w(\AVVI)\subset K$,  we have 
\begin{equation}\label{quasiface}
\Sol^w(\VVI)=\bigcup_{\alpha\subset I}\left[ \Sol^w(\VVI)\cap \F_{\alpha}\right], \ \Sol(\VVI)=\bigcup_{\alpha\subset I}\left[ \Sol(\VVI)\cap \F_{\alpha}\right].
\end{equation}

A topological space $S$ is said to be \textit{connected} if it cannot be represented as $S=U\cup V$, where $U$ and $V$ are disjoint nonempty open sets of $S$. A nonempty subset $A\subset S$ is said to be a \textit{connected component} of $S$ if $A$, equipped with the induced topology, is connected and it is not a proper subset of any connected subset of $S$. The cardinal number of the set of connected components of $S$ is denoted by $\chi(S)$. 

\begin{remark} Let $m$ denote the number of criteria in $(\AVVI)$ and $p$ stand for the number of linear constraints defining $K$. Then, according to \cite[Theorem~3.3]{Hieu16}, the estimate
	\begin{equation}\label{affine_bounds}
\max\left\lbrace \chi(\Sol(\AVVI)),\chi(\Sol^w(\AVVI)) \right\rbrace \leq 2\times3^{2m+2n+3p+1}
\end{equation} holds.
\end{remark}

We are interested in evaluating the number of connected components in the solution sets of monotone AVVIs. Several upper bounds which are strictly smaller than that in \eqref{affine_bounds} will be established.

\section{Bicriteria skew-symmetric affine vector variational inequalities}

The class of skew-symmetric AVVIs encompassed the class of the AVVIs used  for studying  linear fractional vector optimization problems. In addition, the topological structure of the solution sets of skew-symmetric AVVIs is richer and more interesting than that of symmetric monotone AVVIs.

This section is devoted to bicriteria skew-symmetric AVVIs. We will obtain several estimates for the number of connected components of the solution sets of the problems.
\subsection{Unconstrained problems}
We say that $(\AVVI)$ is \textit{nondegenerate} if there exists at least one parameter $\xi\in\Delta_m$ such that
$\det M_{\xi}\neq 0$.

The proof of next result is based on the Cayley Theorem saying that the determinant of a skew-symmetric matrix $A=(a_{ij})$ is the square of the Pfaffian of $A$, which a polynomial of the entries $a_{ij}$.

\begin{theorem}\label{skew}
Consider the problem $(\AVVI)$ with $m=2$, $n\geq 2$, and $n$ being even. If the affine vector variational inequality is unconstrained, skew-symmetric, and nondegenerate, then
	\begin{equation}\label{skew1}
\max\left\lbrace \chi(\Sol(\AVVI)),\chi(\Sol^w(\AVVI)) \right\rbrace \leq n+1.
	\end{equation}
\end{theorem}
\begin{proof} For any $\xi=(\xi_1,\xi_2)\in \Delta_2$, one has $\xi_2=1-\xi_1$ with $\xi_1\in [0,1]$. So, $$\det M_{\xi}=\det\left(\xi_1M_1+(1-\xi_1)M_2\right)$$ is a polynomial in the variable $\xi_1$. Since $M_{\xi}$ is skew-symmetric, according to the Cayley Theorem (see \cite[p.~305]{Lax2007}), there is a polynomial $\pf(\xi_1)$ in the variable $\xi_1$ such that 
	\begin{equation}\label{pfaff}
	\det M_{\xi}=\pf^2(\xi_1).
	\end{equation}
Since $(\AVVI)$ is nondegenerate by our assumption, the polynomial $\pf(\xi_1)$ is not identically null. Suppose that the equation $\pf(\xi_1)=0$ has $k\geq 0$ distinct solutions $\xi_1^1,\xi_1^2,\cdots,\xi_1^{k}$ on $[0,1]$ with $0\leq \xi_1^1<\xi_1^2<\cdots <\xi_1^{k}\leq 1.$ Since the equation (\ref{pfaff}) implies with $\deg\pf(\xi_1)\leq n/2$, we have  $k\leq n/2$. For $$\mathcal{D}:=\{\xi_1\in [0,1]\;:\;\pf(\xi_1)\neq 0\},$$ one has $\chi(\mathcal{D})\leq k+1$. 	
Since  the affine vector variational inequality under consideration is unconstrained, according to Remark \ref{remark_intK}, the multifunction $$\Psi: [0,1]\rightrightarrows\R^n,\quad  \xi_1\mapsto \Phi((\xi_1,1-\xi_1))=\Sol(\VI)_{(\xi_1,1-\xi_1)},$$ is single-valued and continuous on $\mathcal{D}$. Therefore, we obtain $\chi(\Psi(\mathcal{D}))\leq \chi(\mathcal{D})$. 

According to Theorem \ref{thm_scalarization}, one has \begin{equation}\label{skewsolw}
	\Sol^w(\AVVI)=\Psi(\mathcal{D})\cup\Psi(\xi^1_1)\cup\cdots\cup\Psi(\xi^k_1).
	\end{equation}
As $(\AVVI)$ is monotone by our assumption, the sets $\Psi(\xi^i_1)$, $i=1,\dots,k$, are convex. Therefore, the number of connected components of the set in the right-hand side of
\eqref{skewsolw} does not exceed $\chi(\mathcal{D}) +k$. Thus, $$\chi(\Sol^w(\AVVI))\leq \chi(\mathcal{D}) +k\leq 2k+1.$$
Since $k\leq n/2$, this yields $\chi(\Sol^w(\AVVI))\leq n+1.$ 

The inequality $\chi(\Sol(\AVVI))\leq n+1$ can be proved similarly. The only change is that instead of $\xi=(\xi_1,1-\xi_1)\in\Delta_2$ one considers $\xi=(\xi_1,1-\xi_1)\in\ri\Delta_2$. 

Combining the obtained estimates, we get \eqref{skew1}.
\qed
\end{proof}

\begin{remark} Concerning the estimate \eqref{skew1}, we observe that
no \textit{bicriteria} unconstrained skew-symmetric AVVI whose weak Pareto solution set (or, Pareto solution set) possesses $n+1$ connected components has been known in the literature, so far. 
\end{remark}

Theorem \ref{skew} shows that $n+1$ is an upper estimate for the number of connected components of the solution sets of any \textit{bicriteria unconstrained skew-symmetric} AVVI \textit{with $n$ being even}. The following theorem tells us that such an upper estimate cannot be smaller than $\frac{n}{2}+1$.

\begin{theorem}\label{even} For each even number $n\geq 2$, there exists a bicriteria unconstrained skew-symmetric affine vector variational inequality, denoted by $(\Pone)$, such that
	$$\chi(\Sol(\Pone))= \chi(\Sol^w(\Pone))= \frac{n}{2}+1.$$
\end{theorem}
\begin{proof} To prove the statement, fix any even number $n\geq 2$. Consider a bicriteria AVVI, denoted by $(\Pone)$, where the  $n\times n$ matrices $M_1,M_2$ are the following anti-diagonal and skew-symmetric matrices 
\begin{equation}\label{matrixM12}M_1=\begin{bmatrix}
	&  &  &&  & \ 1  \\ 
	&   O& & 	  & \ \iddots   &	\\ 
	&     && \ 1   &   &\\ 
	&    &-1 & &  &	\\ 
&  \iddots   & && O  &\\ 
	-1& &&& &  \\ 
	\end{bmatrix}, \  \ M_2=\begin{bmatrix}
	&  &  &&  &-1  \\ 
	&   O& & 	  & \ \iddots   &	\\ 
	&     &&-s   &   &\\ 
	&    & \ s & &  &	\\ 
	& \ \iddots   & && O  &\\ 
	\ 1& &&& &  \\ 
	\end{bmatrix}\end{equation}
with $s=n/2$, and the vectors $q_i=(q_{i1},\dots,q_{in})^T\in \R^{2s}$, $i=1,2$, are defined by
	\begin{equation}\label{matrixq}
q_{1j}=q_{2j}=
-1, \ j=1,...,n.
	\end{equation}
For each $\xi=(\xi_1,1-\xi_1)\in \Delta_2$ with $\xi_1\in [0,1]$,  $M_{\xi}=\xi_1M_1+(1-\xi_1)M_2$ and $q_{\xi}=\xi_1q_1+(1-\xi_1)q_2$ can be rewritten as
	$$M_{\xi}=\begin{bmatrix}
	&  &  &&  & 2\xi_1-1  \\ 
	&   O& & 	  & \iddots \ \   &	\\ 
	&     && (s+1)\xi_1-s   &   &\\ 
	&    &s-(s+1)\xi_1& &  &	\\ 
	& \ \ \iddots   & && O  &\\ 
	1-2\xi_1& &&& &  \\ 
	\end{bmatrix}, \ 
	q_{\xi}=\begin{bmatrix}
	-1 \\ 
	\vdots\\
	-1\\
	-1\\
	\vdots\\
	-1\\
	\end{bmatrix}.$$
The determinant  of $M_\xi$ is $$p(\xi_1)=(2\xi_1-1)^2(3\xi_1-2)^2...((s+1)\xi_1-s)^2.$$
The equation $p(\xi_1)=0$ has $s$ distinct solutions  \begin{equation}\label{xi_k}
\xi_1^k=\frac{k}{k+1},\quad k=1,...,s.
\end{equation}
For $\mathcal{D}:=\{\xi_1\in [0,1]\,:\,p(\xi_1)\neq 0\}$, we have $$\mathcal{D}=\mathcal{D}_0\cup \mathcal{D}_1\cup \cdots \cup \mathcal{D}_s,$$ where $\mathcal{D}_0=\left[0 ,\xi_1^1\right)$, $\mathcal{D}_k=(\xi_1^k,\xi_1^{k+1})$ with $k=1,...,s-1$ and $\mathcal{D}_{s}=(\xi_1^s,1].$
	
According to Remark \ref{remark_intK}, $x\in \R^{2s}$ solves $(\AVI)_{\xi}$ if and only if $M_{\xi}x=-q_{\xi}$. This means that $x$ is a solution of the system
	\begin{equation}\label{system-f0}
	\begin{bmatrix}
	&  &  &&  & 2\xi_1-1  \\ 
	&   O& & 	  & \iddots \ \   &	\\ 
	&     && (s+1)\xi_1-s   &   &\\ 
	&    &s-(s+1)\xi_1& &  &	\\ 
	& \ \ \iddots   & && O  &\\ 
	1-2\xi_1& &&& &  \\ 
	\end{bmatrix} \begin{bmatrix}
	x_1 \\ 
	\vdots\\
	x_s\\
	x_{s+1}\\
	\vdots\\
	x_{2s}
	\end{bmatrix}=\begin{bmatrix}
	1\\ 
	\vdots\\
	1\\
	1\\
	\vdots\\
	1
	\end{bmatrix}. 
	\end{equation}

To calculate the weak Pareto solution set of $(\Pone)$, we will solve the system~\eqref{system-f0} in the following two cases. 
\begin{itemize}
	\item[(a)] $\xi_1\notin \mathcal{D}$. This means $\xi_1=\xi_1^k$ for $k=1,...,s$, then the $k$-th equation of~\eqref{system-f0} becomes $0x_{2s+1-k}=1$. Thus,~\eqref{system-f0} has no solution.
	\item[(b)] $\xi_1\in \mathcal{D}$. The system \eqref{system-f0} leads to 
	\begin{equation}\label{xi_formula}x_i=-x_{2s+1-i}=\dfrac{1}{i-(i+1)\xi_1},\quad \ i=1,...,s.\end{equation}
Denote this unique solution $x=x(\xi_1)$ of \eqref{system-f0} by $g(\xi_1)$. It is easy to see that
	\begin{equation}\label{g_formula}
	g(\xi_1)=\frac{1}{1-2\xi_1}u_1+\frac{1}{2-3\xi_1}u_2+\cdots+\frac{1}{s-(s+1)\xi_1}u_s,
	\end{equation}
where 
\begin{equation}\label{u_formula}
u_1=(1,0,...,0,-1), u_2=(0,1,...,-1,0),..., u_s=(0,...,1,-1,...,0).
\end{equation}
\end{itemize}
Thanks to Theorem \ref{thm_scalarization}, we obtain
$$\Sol^w(\Pone)=g(\mathcal{D})=g(\mathcal{D}_0)\cup g(\mathcal{D}_1)\cup\cdots\cup g(\mathcal{D}_s).$$ Let us prove that the mapping $g:\mathcal{D}\to \Sol^w(\Pone)$ is a  homeomorphism. Indeed, on each $\mathcal{D}_k$, all the functions $$f_i(\xi_1):=\frac{1}{i-(i+1)\xi_1},\quad i=1,...,s$$
 are continuous. Thus, by \eqref{g_formula} and \eqref{u_formula}, $g$ is continuous on $\mathcal{D}$. The fact that $g$ is injective on  $\mathcal{D}$ can be easily verified. Hence, $g$ is a bijection between $\mathcal{D}$ and $g(\mathcal{D})$. 
If $x\in g(\mathcal{D})$, then $x=(x_1,...,x_s,-x_s,...,-x_1)$, where the coordinates $x_i$ are given by \eqref{xi_formula}. In particular, $x_1=\dfrac{1}{1-2\xi_1}$. Hence,
	 $$\xi_1=h(x_1,...,x_n):=\frac{x_1-1}{2x_1}.$$ As $h$ is obviously continuous on $g(\mathcal{D})$ and it is the inverse of $g$, we conclude that $g:\mathcal{D}\to \Sol^w(\Pone)$ is a homeomorphism. Therefore, we have $$\chi(\Sol^w(\Pone))=\chi(\mathcal{D})=s+1.$$
	
Clearly, $\Sol(\Pone)=\Sol^w(\Pone)\setminus\{g(0),g(1)\}$ where  $$g(0)=(1,...,\frac{1}{s},-\frac{1}{s}...,-1),\quad g(1)=(-1,...,-1,1,...,1).$$
The above argument shows that $g:\mathcal{D}\setminus\{0,1\}\to \Sol(\Pone)$ is a homeomorphism. Then $\Sol(\Pone)$ also has $s+1$ connected components. \qed
\end{proof}
\begin{remark} In \cite{YY2011}, there is a conjecture that the maximal number of connected components of the solution sets of a monotone AVVI equals to $m$ - the number of the criteria in the problem. Theorem \ref{even} shows that the conjecture is wrong when $m=2$ and $n\geq 4$ because, in that case,  $\frac{n}{2}+1\geq 3>m$.
\end{remark}
\begin{remark} As shown by Yen and Phuong \cite{Yen_Phuong_2000}, the necessary and sufficient condition for a point to be a Pareto solution of a linear fractional vector optimization problem (LFVOP) can be regarded as the condition for that point to be a Pareto solution of a skew-symmetric AVVI. In \cite{Yen_Phuong_2000}, the authors also observed that the necessary and sufficient condition for a point to be a weak Pareto solution of a linear fractional vector optimization problem (LFVOP) can be regarded as the condition for that point to be a weak Pareto solution of a skew-symmetric AVVI. For any bicriteria LFVOP, the intersection of the Pareto solution set with any pseudo-face of the polyhedral convex constraint set is a convex set. This result implies that the Pareto solution set of problem  $(\Pone)$ in the proof of Theorem \ref{even} does coincide with the Pareto solution set of any bicriteria LFVOP with $n\geq 2$. Indeed, since  $(\Pone)$ is an unconstrained bicriteria AVVI, the constraint set is $K=\R^n$. Hence the unique pseudo-face of $K$ is $\R^n$. So, if $\Sol(\Pone)$  coincides with the Pareto solution set of a bicriteria LFVOP with $K=\R^n$ and $n\geq 2$, then $\Sol(\Pone)\cap \R^n=\Sol(\Pone)$ must be convex. But, this is not true. Thus, the question whether the maximal number of connected components of the solution sets of a LFVOP with $m$ criteria cannot exceed $m$, which was raised in \cite{HPY_2002}, remains unresolved. 
\end{remark}
\subsection{Constrained problems}
Recall that the problem $(\Pone)$ in the proof of Theorem \ref{even} is unconstrained. There is a natural question: \textit{What happens with the maximal number of connected components of the solution sets of a monotone AVVI when the number of constrains increase?} Observe that, if $K$ is compact, then both Pareto solution set and weak Pareto solution set of the given monotone AVVI are connected \cite[Theorem 4.1]{YY2011}. 

\begin{theorem}\label{evenn1} For any even number $n\geq 2$ and integer $p$ with $1 \leq p \leq n/2$, there exists a bicriteria skew-symmetric affine vector variational inequality, denoted by  $(\Pa_p)$, such that
$$\chi(\Sol(\Pa_p))=  \chi(\Sol^w(\Pa_p))= \frac{n}{2}+p+1.$$
\end{theorem}

\begin{proof} Suppose that $n=2s$ with $s\geq 1$. For each $1 \leq p \leq s$, consider the bicriteria skew-symmetric AVVI $(\Pa_p)$ with $M_1,M_2$ described in \eqref{matrixM12}, $q_1,q_2$ given by \eqref{matrixq}, and the constrain set $K_p$ is defined by
	$$K_p=\{x\in \R^n:-x_1-x_{2s}\geq -1, \cdots, -x_p-x_{2s+1-p}\geq -1\}.$$
We will show that  $\chi(\Sol^w(\Pa_p))=\chi(\Sol(\Pa_p))=s+p+1$. 

\textit{Since the case $p=s$ is the most complicated, we will deal with it. The cases where $p\in\{1,\dots, s-1\}$ can be treated similarly.}

The constraint $Ax\geq b$ describing $K_s$, with $A\in \R^{s\times 2s}$ and $b\in \R^{s\times 1}$, can be rewritten as 
	$$	Ax=\begin{bmatrix}
-1 &  & O & O&  &-1 \\ 
	 &  \ddots  &   & &\iddots& \\ 
	O&   & -1   &-1& & O\\ 
	\end{bmatrix}\begin{bmatrix}
	x_1 \\ 
	\vdots\\
	x_{2s}
	\end{bmatrix}\geq \begin{bmatrix}
	-1 \\ 
	\vdots\\
	-1
	\end{bmatrix}=b.$$
In accordance with Theorem \ref{thm_scalarization} and formula \eqref{quasiface}, we will find solutions of $(\Pa_s)$ on each pseudo-face $\F_{\alpha}$ of $K_s$ by using the formula
 $$\Sol^w(\Ptwo)\cap \F_{\alpha}=\bigcup_{\xi\in \Delta_2}({\Sol(\AVI)}_\xi\cap \F_{\alpha}).$$
 
\noindent \textsc{Claim 1.} \textit{One has $\Sol^w(\Pa_s)\cap \F_{\emptyset}=g(\mathcal{D})$, where $\mathcal{D}$ is defined as in the proof of Theorem $\ref{even}$ and $g$ is given by \eqref{g_formula}.}

\smallskip
Indeed, by Remark \ref{remark_intK}, $x\in \Sol^w(\Pa_s) \cap \F_{\emptyset}$ if and only if $x\in\F_{\emptyset}$ and \eqref{system-f0} is satisfied. The arguments in the proof of Theorem \ref{even} can be used again. In particular, if $\xi_1\notin \mathcal{D}$, then system \eqref{system-f0} has no solution. If $\xi_1\in \mathcal{D}$, then $x=g(\xi_1)$ satisfies \eqref{system-f0} and the conditions $-x_i-x_{2s+1-i}>-1$, $i=1,...,s$. Thus, the formula $\Sol^w(\Pa_s)\cap \F_{\emptyset}=g(\mathcal{D})$ is valid.

\smallskip
\noindent \textsc{Claim 2.} \textit{For each $k\in\{1,...,s\}$, one has
\begin{equation}\label{Si}
	\Sol^w(\Pa_s)\cap \F_{\{k\}}=\left\lbrace e_{2s+1-k}+ tu_k+\sum_{1\leq j \leq s}^{j\neq k}\beta_{kj}u_j, t\in \R\right\rbrace=:S_k,
	\end{equation}
where the vectors $u_j$ are given in \eqref{u_formula}, $e_i$ is the $i$-th unit vector of $\mathbb R^n$, and
$$\beta_{kj}:=\frac{1}{j-(j+1)\xi_1^k},\quad \ j\in\{1,...,s\}\setminus\{k\}.$$}

\smallskip
We only give the proof for the case $k=1$, because other cases can be treated similarly. Note that
	$$\F_{\{1\}}=\{x\in \R^n:-x_1-x_{2s}= -1,-x_j -x_{2s+1-j}> -1, \ 2\leq j\leq s\}.$$ Applying Theorem \ref{Lagrange multipliers rule} to the affine variational inequality $(\AVI)_\xi$, we see that the second equation in
\eqref{system} gives $\lambda_j =0$ for all $j=2,..., s$. Meanwhile, the equation $M_{\xi}x-A^T\lambda+q_{\xi} =0$, which corresponds to the first relation in \eqref{system}, can be written as
\begin{equation}\label{system-f1}
	\begin{bmatrix}
O&  &  &&  & 2\xi_1-1\\ 
&   & & 	  & 3\xi_1-2 &	\\ 
&    &\iddots  & &  &	\\ 
&  2-3\xi_1& &&   &\\ 
1-2\xi_1& &&& &  O\\ 
\end{bmatrix} \begin{bmatrix}
x_1 \\ 
x_2 \\ 
\vdots\\
x_{2s-1}\\
x_{2s}
\end{bmatrix}=\begin{bmatrix}
1-\lambda_1\\ 
1\\
\vdots\\
1\\
1-\lambda_1 \\
\end{bmatrix}.
\end{equation}

Now, let the polynomial $p(\xi_1)$ and its roots $\xi_1^k$, $k=1,...,s$, be defined as in the proof of Theorem \ref{even}. In accordance with \eqref{xi_k}, we have $\xi_1^1=1/2$. Consider the following two cases.

\smallskip
\noindent \textsc{Case 1:} $\xi_1 \neq \xi_1^1$. The first and the last equations  of \eqref{system-f1} yield $x_{1}=-x_{2s}$. So, $-x_1-x_{2s}\neq-1$. This means that \eqref{system-f1} has no solution in $\F_{\{1\}}$. 

\smallskip
\noindent \textsc{Case 2:} $\xi_1=\xi_1^1$. Clearly, if \eqref{system-f1} has a solution $x=(x_1,...,x_n)\in\mathbb R^n$, then one must have $\lambda_1=1$. For $\lambda_1=1$, \eqref{system-f1} is equivalent to saying that
	$$x_j=-x_{2s+1-j}=\beta_{1j}=\dfrac{1}{j-(j+1)\xi_1^1},  \ j=2,...,s,$$ and  $x_1=t,\, x_{2s}=t'$, where $t,t'\in \R$ can be chosen arbitrarily.
	If $x\in\F_{\{1\}}$, then $-x_1-x_{2s}=-1$. Therefore, 
	$$\Sol^w(\Pa_s)\cap \F_{\{1\}}=\left\lbrace e_{2s}+ tu_1+\sum_{2\leq j \leq s}\beta_{1j}u_j, \ t\in \R\right\rbrace=S_1.$$

The claim is justified.

\smallskip
\noindent \textsc{Claim 3.} \textit{We have $\Sol^w(\Pa_s)\cap \F_{\alpha}=\emptyset$ for all index set $\alpha$ with $|\alpha|\geq 2$.}
\smallskip

We only give the proof for the case $\alpha=\{1,2\}$, because other cases can be analyzed in the same way. The pseudo-face $\F_{\{1,2\}}$ of $K_s$ is given by
	$$\{x\in \R^n:-x_1-x_{2s}= -1,-x_2-x_{2s-1}= -1,-x_j -x_{2s+1-j}> -1, j\neq 1,2\}.$$
Applying Theorem \ref{Lagrange multipliers rule} to $(\AVI)_\xi$ gives $\lambda_j =0$ with $3 \leq j\leq s$. In addition, the equation $M_{\xi}x-A^T\lambda+q_{\xi} =0$, which corresponds to the first relation in \eqref{system}, can be written as
\begin{equation}\label{system-f2}
	\begin{bmatrix}
O&  &  &&  & 2\xi_1-1\\ 
&   & & 	  & 3\xi_1-2 &	\\ 
&    &\iddots  & &  &	\\ 
&  2-3\xi_1& &&   &\\ 
1-2\xi_1& &&& &  O\\ 
\end{bmatrix}\begin{bmatrix}
x_1 \\ 
x_2 \\ 
\vdots\\
x_{2s-1}\\
x_{2s}
\end{bmatrix}=\begin{bmatrix}
1-\lambda_1\\
1-\lambda_2 \\ 
\vdots\\
	1-\lambda_2 \\ 
	1-\lambda_1\\
	\end{bmatrix}. 
	\end{equation}

Suppose that $x=(x_1,...,x_{2s})$ is a solution of \eqref{system-f2}. For a value $\xi_1\in [0,1]$, since the situations $\xi_1=\xi_1^1$ and $\xi_1=\xi_1^2$ cannot occur simultaneously, one must have either $\xi_1\neq \xi_1^1$, or $\xi_1\neq \xi_1^2$. If $\xi_1\neq \xi_1^1$, then the first and the last equations  of \eqref{system-f2} force $x_{1}=-x_{2s}$. So, $-x_1-x_{2s}\neq-1$. This means that  $x\notin\F_{\{1,2\}}$. If $\xi_1\neq \xi_1^2$, then the second and the $(2s-1)$-th equations of \eqref{system-f2} imply $x_2=-x_{2s-1}$. So, $-x_2-x_{2s-1}\neq-1$. This means that $x\notin\F_{\{1,2\}}$. Thus, the intersection of the weak Pareto solution set with $\F_{\{1,2\}}$ is empty. 

Let $\mathcal{D}_0,\, \mathcal{D}_1,...,\mathcal{D}_s$ be defined as in the proof of Theorem \ref{even}. Recalling that $\mathcal{D}=\mathcal{D}_0\cup \mathcal{D}_1\cup \cdots \cup \mathcal{D}_s$ and summarizing all the above, we obtain
\begin{equation}\label{solwps}
	\Sol^w(\Pa_s)=g(\mathcal{D}_0)\cup\cdots\cup g(\mathcal{D}_s)\cup  S_1\cup\cdots\cup S_s.
	\end{equation}
From \eqref{Si} it follows that $S_1, ...,S_s$ the straight lines. Since the  pseudo-faces $\F_{\{k\}}$, $k=1,...,s$, of $K_s$ are pairwise disjoint, the straight lines $S_1, ...,S_s$ are disjoint. Consider the hyperplane $$H_i:=\{x\in\R^n:-x_i-x_{2s+1-i}=-1\}\quad (i=1,...,s).$$ By \eqref{g_formula} and an elementary calculation, we can show that the distance from the point $g(\xi_1)$ to each $H_i$, $i=1,...,s$, is $\frac{\sqrt{2}}{2}$. Since $S_i\subset H_i$, it follows that the distance from every point of $g(\mathcal{D})$ to $S_1\cup\cdots\cup S_s$ is greater or equal $\frac{\sqrt{2}}{2}$. In the proof of Theorem \ref{evenn1}, we have shown that $g(\mathcal{D}_0),..., g(\mathcal{D}_s)$ are connected components of $g(\mathcal{D})$. Consequently, every member of the union of sets on the right-hand side of \eqref{solwps} is a connected component of $\Sol^w(\Pa_s)$. It follows that $$\chi(\Sol^w(\Pa_s))=(s+1)+s=n+1.$$
	
It remains to prove that $\chi(\Sol(\Pa_s))= n+1.$ By Theorem \ref{thm_scalarization}, 
	\begin{equation*} \Sol(\Pa_s)=\bigcup_{\xi\in{\ri}\Delta_2}{\Sol(\VI)}_\xi. \end{equation*} So, for every pseudo-face $\F_{\alpha}$ of $K_s$, one has
	$$\Sol(\Ptwo)\cap \F_{\alpha}=\bigcup_{\xi\in \ri\Delta_2}({\Sol(\AVI)}_\xi\cap \F_{\alpha}).$$
	Therefore, looking back to the proofs of the above Claims 1 and 2 justifies the following results:
	
	\smallskip
	\noindent \textsc{Claim 1'.} \textit{It holds that $\Sol(\Pa_s)\cap \F_{\emptyset}=g(\mathcal{D})\setminus\{g(0),g(1)\}$.}
	
	\smallskip
	\noindent \textsc{Claim 2'.} \textit{For each $k\in\{1,...,s\}$, $\Sol(\Pa_s)\cap \F_{\{k\}}=S_k$.}
	
	\smallskip
	Since $\Sol(\Pa_s)\subset \Sol^w(\Pa_s)$, combining these facts with the above Claim 3, we get $\Sol(\Pa_s)=\Sol^w(\Pa_s)\setminus\{g(0),g(1)\}$. Hence, $\chi(\Sol(\Pa_s))=n+1.$ \qed
\end{proof}

\begin{remark} Applied to the case of the problem $(\Pa_p)$, the argument that was used for dealing with $(\Pa_s)$ in the proof of Theorem \ref{evenn1} yields
	$$\Sol^w(\Pa_p)=g(\mathcal{D}_0)\cup\cdots\cup g(\mathcal{D}_s)\cup  S_1\cup\cdots\cup S_p,$$ where the straight lines $$S_k=\left\lbrace e_{2s+1-k}+ tu_k+\sum_{1\leq j \leq s}^{j\neq k}\beta_{kj}u_j, t\in \R\right\rbrace\quad (k=1,...,s)$$ are given by \eqref{Si}. The Pareto solution set is
	$$\Sol(\Pa_p)=\left( g(\mathcal{D}_0)\cup\cdots\cup g(\mathcal{D}_s)\cup  S_1\cup\cdots\cup S_p\right)\setminus\{g(0),g(1)\}.$$
 Both solution sets have $s+p+1$ connected components.
\end{remark}

\section{The maximum of the number of connected components}

Denote by $\chi(m,n,p)$ the \textit{maximum} of the numbers of connected components of the Pareto solution sets of monotone AVVIs which have $n$ variables, $m$ criteria, and $p$ linear constraints. For the weak Pareto solution sets, the corresponding number is denoted by $\chi^w(m,n,p)$.

\begin{remark} According to \eqref{affine_bounds}, both numbers $\chi(m,n,p)$ and $\chi^w(m,n,p)$ are less or equal to $2\times3^{2m+2n+3p+1}$. Meanwhile, Theorem \ref{even} and Theorem \ref{evenn1}  state that for any even number $n\geq 2$ and integer $p$ with $0 \leq p \leq n/2$, 
$$\frac{n}{2}+p+1\leq \min\left\{\chi(2,n,p), \chi^w(2,n,p)\right\}.$$\end{remark}

\begin{remark} Theorem 9.13 in \cite{Yen2012} shows that there exists a linear fractional vector optimization problem $(\LFVOP)$ whose Pareto solution set coincides with the weak Pareto solution set and has exactly $m$ connected components. We observe that the skew-symmetric AVVI expressing the necessary and sufficient optimality of this problem $(\LFVOP)$ has $m=n$ and $p=n+1$.
Hence, for $n\geq 2$, we have
$$n\leq \min\left\lbrace \chi(n,n,n+1), \chi^w(n,n,n+1)\right\rbrace.$$
\end{remark}

Now, let us investigate the change of $\chi(m,n,p)$ and of $\chi^w(m,n,p)$) w.r.t. the increase of each one of the parameters $n$ and $m$.

\begin{proposition}\label{nondecreasingn} For fixed parameters $m\geq 1$ and $p\geq 0$, the sequences $\{\chi(m,n,p)\,:\,n\in \N\}$ and $\{\chi^w(m,n,p):n\in \N\}$ are increasing.
\end{proposition}
\begin{proof} Let $m\geq 1$ and $p\geq 0$ be fixed. We will prove the assertion for the first sequence, because the result for the second one can be obtained in a similar manner. It suffices to show that $$\chi(m,n,p) \leq\chi(m,n+1,p)\ \; \forall n\in \N.$$
Consider a monotone AVVI, denoted by $(\Pa)$, which is defined by some operators $F_i(x)$ for $ i=1,...,m,$ and a constraint set $K_P\subset \R^n$. Suppose that $\chi(\Sol(\Pa))=q$. Based on the problem $(\Pa)$, we will construct another monotone AVVI, denoted by $(\Q)$, which has ${n+1}$ variables such that $\chi(\Sol(\Q))=q$. Namely, let the AVVI problem $(\Q)$ be defined by the operators
	$$G_i(x_1,...,x_{n},x_{n+1}):=F_i(x_1,...,x_{n}),\ \; i=1,...,m,$$ and the constraint set $$K_Q:=\{x'\in \R^{n+1}\;:\;x'=(x_1,...,x_{n},0),\ (x_1,...,x_{n})\in K_P\}.$$
	For each $i=1,...,m$, the operator $G_i$ is affine and monotone on $K_Q$. 	
Let $(\Pa_{\xi})$ (resp., $(\Q_{\xi})$), where $\xi \in \ri\Delta_m$,  be the parametric variational inequality corresponding to  $(\Pa)$ (resp., $(\Q)$). We claim that $x\in \Sol(\Pa_{\xi})$ if and only if $x'\in \Sol(\Q_{\xi})$, where $x'=(x,0)\in \R^{n+1}$. Indeed, since
	$$\left\langle \sum_{i=1}^m\xi_iF_i(x),y-x\right\rangle=\left\langle \sum_{i=1}^m\xi_iG_i(x'),y'-x'\right\rangle\ \;  \forall y\in K_P, \ \forall y'\in K_Q,$$ two conditions 
	 $$	\left\langle  \sum_{i=1}^m\xi_iF_i(x),y-x\right\rangle  \geq 0,\ \; \forall y\in
	K_P,$$
and 
$$\left\langle \sum_{i=1}^m\xi_iG_i(x'),y'-x'\right\rangle  \geq 0,\ \; \forall y'\in K_Q,$$ are equivalent. 
	
Consider the map $h:\R^{n} \to \R^{n+1}$ defined by $h(x_1,...,x_{n}) =(x_1,...,x_{n},0).$ From the last observation one has $h(\Sol(\Pa)_\xi)=\Sol(\Q)_\xi$. Therefore, according to Theorem \ref{thm_scalarization}, one has $h(\Sol(\Pa))=\Sol(\Q)$. As the induced map $h:\R^{n} \to h(\R^{n})$ is a homeomorphism,  the equality  $\chi(\Sol(\Q))=\chi(\Sol(\Pa))=q$ is valid.
By the definition of $\chi(m,n,p)$, one has $\chi(m,n,p) \leq \chi(m,n+1,p)$. \qed
\end{proof}

\begin{proposition}\label{nondecreasingm}  For fixed parameters $n\geq 1$ and $p\geq 0$, the sequences $$\{\chi(m,n,p)\;:\;m\in \N\}\quad {and}\quad \{\chi^w(m,n,p)\;:\;m\in \N\}$$ are increasing. 
\end{proposition}
\begin{proof}  Let  $n\geq 1$ and $p\geq 0$ be fixed. We will prove that the first sequence is increasing, because the assertion about the second sequence can be proved similarly. To show that $$\chi(m,n,p) \leq\chi(m+1,n,p)\quad \forall m\in \N,$$
we consider a monotone AVVI, denoted by $(\Pa)$, with the affine operators $F_i(x)$, $ i=1,...,m,$ and the polyhedral convex set $K_P\subset \R^n$ being given arbitrarily. Suppose that $\chi(\Sol(\Pa))=q$. Based on $(\Pa)$, we can build another monotone AVVI problem $(\Q)$ having $m+1$ criteria, such that $\chi(\Sol(\Q))=q$. Indeed, let $(\Q)$ be defined by the operators
\begin{equation}\label{solution_eta0}
G_i(x):=\left\{\begin{array}{ll}
F_i(x)& \quad \hbox{ if } \ \ i=1,...,m,  \\
F_{m}(x) & \quad \hbox{ if } \ \ i=m+1,  \\
\end{array}\right.
\end{equation}
and the constraint set $K_Q:=K_P$. Clearly, all the operators $G_i(x)$ are affine and monotone on $K_Q$.
	
For any $\xi \in \ri\Delta_m$ and $\eta\in \ri\Delta_{m+1}$, let  $(\Pa_{\xi})$ and $(\Q_{\eta})$, respectively, denote the parametric variational inequality problems corresponding to  $(\Pa)$ and $(\Q)$. According to Theorem \ref{thm_scalarization}, 
	\begin{equation}\label{new_equation}
	\Sol(\Pa)=\bigcup_{\xi\in\ri\Delta_m}\Sol(\Pa_\xi), \ \ \Sol(\Q)=\bigcup_{\eta\in\ri\Delta_{m+1}}\Sol(\Q_{\eta}).
	\end{equation}
We will show that $\Sol(\Pa)=\Sol(\Q)$. For every $x\in\Sol(\Pa)$, by \eqref{new_equation} we find $\xi\in \ri\Delta_m$ such that $x\in \Sol(\Pa_{\xi})$. Then,
\begin{equation}\label{solution_eta}
\left\langle  \sum_{i=1}^m\xi_iF_i(x),y-x \right\rangle \geq 0  \ \ \forall y\in K_P.
\end{equation}
Note that
	$$\sum_{i=1}^m\xi_iF_i(x)=\sum_{i=1}^{m-1}\xi_iF_i(x)+\frac{\xi_m}{2}F_m(x)+\frac{\xi_m}{2}F_m(x).$$
Define $\eta=(\eta_1,\dots,\eta_{m+1})\in \ri\Delta_{m+1}$ by setting
$$\eta_i=\left\{\begin{array}{ll}
\xi_i& \quad \hbox{ if } \ \ i=1,...,m-1,  \\
\dfrac{\xi_m}{2}& \quad \hbox{ if } \ \ i=m,m+1.  \\
\end{array}\right.$$
Then, from \eqref{solution_eta} we get
$$\left\langle  \sum_{i=1}^{m+1}\eta_iG_i(x),y-x \right\rangle \geq 0  \ \ \forall y\in K_Q.$$
This means that $x\in \Sol(\Q_{\eta})$. Hence, one has $x\in \Sol(\Q)$. 

Conversely, for each $x\in\Sol(\Q)$, there exists $\eta\in \ri\Delta_{m+1}$ with $x\in \Sol(\Q_{\eta})$. Therefore,
\begin{equation}\label{solution_eta1}
\left\langle   \sum_{i=1}^{m+1}\eta_iG_i(x),y-x\right\rangle  \geq 0 \ \; \forall y\in K_Q.\end{equation}
From \eqref{solution_eta0} it follows that
	$$\sum_{i=1}^{m+1}\eta_iG_i(x)=\sum_{i=1}^{m-1}\eta_iF_i(x)+(\eta_m+\eta_{m+1})F_m(x).$$
Let $\xi\in \ri\Delta_m$ be given by
$$\xi_i=\left\{\begin{array}{ll}
\eta_i& \quad \hbox{ if } \ \ i=1,...,m-1,  \\
\eta_m+\eta_{m+1}& \quad \hbox{ if } \ \ i=m.  \\
\end{array}\right.$$
From \eqref{solution_eta1} we have 
$$\left\langle  \sum_{i=1}^m\xi_iF_i(x),y-x \right\rangle \geq 0  \ \ \forall y\in K_P.$$
This means that $x\in \Sol(\Pa_{\xi})$.

We have thus proved that $\Sol(\Pa)=\Sol(\Q)$. So, $\chi(\Sol(\Q))=\chi(\Sol(\Pa))=q$. Hence, by the definition of $\chi(m,n,p)$, $\chi(m,n,p) \leq \chi(m+1,n,p)$. 
 \qed
\end{proof}

A lower bound for $\chi(m,n,p)$ and $\chi^w(m,n,p)$ can be obtained, provided that the number of constraints is relatively small in comparison with the number of variables.  

\begin{theorem}\label{lowerbound2} Suppose that the affine vector variational inequality $(\AVVI)$ is monotone with $m\geq 2,\; n\geq 2$ and $\lfloor\frac{n}{2}\rfloor\geq p\geq 0$, where $\lfloor a \rfloor$ is the integer part of real number $a$. Then,
	$$\lfloor\frac{n}{2}\rfloor+p+1\leq \min\left\lbrace \chi(m,n,p), \chi^w(m,n,p)\right\rbrace.$$
\end{theorem}
\begin{proof} Let  $m\geq 2,\; n\geq 2$ and $p\geq 0$ be fixed with $\lfloor\frac{n}{2}\rfloor\geq p\geq 0$. We will prove the inequality \begin{equation}\label{ineq_new} \lfloor\frac{n}{2}\rfloor+p+1\leq  \chi(m,n,p),\end{equation} because the inequality $\lfloor\frac{n}{2}\rfloor+p+1\leq  \chi^w(m,n,p)$ can be proved similarly.

Suppose that $n$ is an even with $n=2s$, with $0 \leq p \leq s$. One has $$s+p+1\leq \chi(2,n,p)\leq \chi(m,n,p)\ \  \forall m\geq 2,$$ where the first inequality holds by Theorem \ref{evenn1} and the second one is valid by Proposition \ref{nondecreasingm}.
Now, suppose that $n$ is an odd number with $n=2s+1$ and $0 \leq p \leq s$. Applying the inequality $s+p+1\leq \chi(m,2s,p)$ and Proposition~\ref{nondecreasingn}, we have $$s+p+1\leq\chi(m,2s,p) \leq \chi(m,2s+1,p) \ \  \forall m\geq 2.$$ Thus, \eqref{ineq_new} is valid. \qed
\end{proof}
\begin{remark}
Theorem \ref{lowerbound2} shows that, for any fixed $m$ and $p$, the sequences $\{\chi(m,n,p):n\in \N\}$ and $\{\chi^w(m,n,p):n\in \N\}$ are unbounded. In other words, the parameter $n$ has a direct effect on the maximal number of connected components of the solution sets of monotone AVVIs.
\end{remark}

We conclude this section by two open questions.

\begin{question} When $n$ and $p$ are fixed,  is it true that $\{\chi(m,n,p):m\in \N\}$ and $\{\chi^w(m,n,p):m\in \N\}$ are bounded sequences, or not?
\end{question}
\begin{question}  When  $m$ and $n$ be fixed, is it true that  $\{\chi(m,n,p):p\in \N\}$ and $\{\chi^w(m,n,p):p\in \N\}$ are bounded sequences, or not?
\end{question}

\begin{acknowledgements}
	The author is indebted to Professor Nguyen Dong Yen for many stimulating conversations.
\end{acknowledgements}

\end{document}